\newif\ifpictures
\picturestrue

\documentclass[12pt]{amsart}
\usepackage{amssymb,amsmath}
 \usepackage{amsopn}
 \usepackage{mathabx}
 \usepackage{xspace}
  \usepackage{hyperref}
 \usepackage[dvips]{graphicx}
\usepackage[arrow,matrix,curve]{xy}
\usepackage {color, tikz, tikz-cd}
\usepackage{wasysym}
\usepackage{multirow}

\numberwithin{equation}{section}
\newtheorem{theorem}{Theorem}
\newtheorem{proposition}[theorem]{Proposition}
\newtheorem{lemma}[theorem]{Lemma}
\newtheorem{corollary}[theorem]{Corollary}

\theoremstyle{definition}

\newtheorem{example}[theorem]{Example}
\newtheorem{definition}[theorem]{Definition}

\newtheorem{remark}[theorem]{Remark}

\numberwithin{theorem}{section}

\newcounter{FNC}[page]
\def\newfootnote#1{{\addtocounter{FNC}{2}$^\fnsymbol{FNC}$%
     \let\thefootnote\relax\footnotetext{$^\fnsymbol{FNC}$#1}}}

\newcommand{\C}{\mathbf{C}}
\newcommand{\N}{\mathbf{N}}

\newcommand{\R}{\mathbf{R}}
\renewcommand{\P}{\mathbf{P}}
\newcommand{\Z}{\mathbf{Z}}

\newcommand{\E}{\mathbb{E}}

\newcommand{ \<}{\langle}
\renewcommand{\>}{\rangle}

\newcommand{\br}{\operatorname{br}}
\newcommand{\intt}{\operatorname{int}}
\newcommand{\disc}{\operatorname{Disc}}
\newcommand{\res}{\operatorname{Res}}

\newcommand\cD{{\ensuremath{\mathcal{D}}}\xspace}

\newcommand\cF{{\ensuremath{\mathcal{F}}}\xspace}

\newcommand\del{{\ensuremath{\partial}}}

\newcommand\CB{{\ensuremath{\mathcal{CB}}}\xspace}

\definecolor{NiceBlue}{rgb}{0.2,0.2,0.75}

\DeclareMathOperator{\conv}{conv}

\DeclareMathOperator{\Sol}{Sol}

\DeclareMathOperator{\vol}{vol}

\DeclareMathOperator{\rank}{rank}

\DeclareMathOperator{\Mon}{Mon}

\newcommand{\excise}[1]{}

\title{A Zariski theorem for monodromy of $A$-hypergeometric systems}

\begin{document}

\author{Jens Forsg{\aa}rd}

\address{Jens Forsg{\aa}rd\\ Universiteit Utrecht\\ Mathematisch Instituut\\ Postbus 80010\\
3508 TA Utrecht\\ The Netherlands}
\email{j.b.forsgaard@uu.nl}


\author{Laura Felicia Matusevich}
\address{Laura Felicia Matusevich\\ Department of Mathematics \\
Texas A\&M University \\ College Station, TX 77843.}
\email{laura@math.tamu.edu}

 
\begin{abstract}
We give conditions under which the monodromy group of an $A$-hyper\-geo\-metric
system is invariant under modifications of the collection of characters $A$.
The key ingredient is a Zariski--Lefschetz type theorem for principal
$A$-determinants. 
\end{abstract}

\maketitle

\section{Introduction}
\label{sec:Intro}

This article concerns the study of $A$-hypergeometric
monodromy. $A$-hypergeometric systems were introduced by Gel'fand,
Graev, Kapranov and Zelevinsky in the late twentieth
century~\cite{GG87,GKZ88,GKZ89} in order to provide a
uniform theory for multivariate hypergeometric functions
(see also~\cite{BMW19}), as well as a bridge to toric
geometry. Through this connection, important concepts such as
canonical series solutions~\cite{SST00} and holonomic rank~\cite{Ado94,Ber11,
  MMW05} can be described in combinatorial terms.

The combinatorial tractability of $A$-hypergeometric systems, combined
with existing deep results on monodromy of classical hypergeometric
functions (see, e.g., \cite{BH89, DM86}), gives hope of a correspondingly
rich theory of $A$-hypergeometric monodromy. However, the literature in this
direction is sparse (but, see~\cite{Beu16, Yos97}).
To understand why this is the case, we recall that,
by definition, the monodromy group of a system of differential
equations is a representation of the fundamental group of the
complement of its singular locus. The singular locus of an
$A$-hypergeometric system is the zero set of a polynomial called the
\emph{principal $A$-determinant} or \emph{full discriminant}. 
Geometrically, this is a union of
discriminantal hypersurfaces.
Computing the fundamental group of the complement of an algebraic
hypersurface is a deep and important question, which is challenging in the case of discriminantal hypersurfaces
 \cite{DL81, L09}. For principal $A$-determinants, we are not
aware of any general results in this direction, and this has
obstructed progress on $A$-hypergeometric monodromy. Consequently, our
first major goal is to prove a Zariski--Lefschetz-type theorem for principal
$A$-determinants.

To make this more precise,  recall that an $A$-hypergeometric system $H_A(\beta)$
is defined by a finite collection $A$
of algebraic characters of the torus $(\C^*)^{1+n}$, and a parameter
vector $\beta \in \C^{1+n}$. (See \S\ref{sec:Monodromy} for a precise definition.) 
Denote by $V_A \subset \C^A$ the singular locus of $H_A(\beta)$, which is independent of $\beta$. 
Here, $\C^A$ denotes the complex affine space of dimension $k=|A|$.

Removing a character from $A$ corresponds to restricting, in the space $\C^A$, 
to a coordinate hyperplane.
To study homotopy groups of complements
of embedded algebraic varieties through intersections with linear spaces 
is classical~\cite{HL73, Lef24, Zar37}, and numerous theorems exist
in the literature; in most instances with some smoothness assumption. 
However, the variety $V_A$ is highly singular.
Typically, its Whitney stratification has reducible non\-empty strata in
each codimension. The combination of a highly singular
variety and a coordinate hyperplane implies that standard Lefschetz and Zariski type 
theorems do not apply. 
As a first main result, we provide a combinatorial condition 
that allows us to add or remove characters and still control the effect on 
homotopy. 

Let $N$ denote the \emph{Newton polytope} (i.e., convex hull) of $A$.
We say that $A$ has an interior point if at least one element of $A$
is interior to $N$. For each face $\Gamma$ of $N$ the \emph{face lattice}
of $A$ relative $\Gamma$ is the affine lattice spanned by the elements of $A \cap \Gamma$. We define
the \emph{face saturation} $A^s$ of $A$ to be the largest subset 
of $N \cap \Z^{1+n}$ such that the face lattices of $A^s$ coincide with
the face lattices of $A$. (See Definition~\ref{def:saturations} for a precise statement).
Generalizing a Zariski-style theorem from~\cite{Bes01}, we conclude the following.

\begin{theorem}
\label{thm:MainFundamentalGroup}
Assume that $A$ has an interior point. Then,
the inclusion $\C^A \rightarrow \C^{A^s}$ given by appending zeros to $x$ for the coordinates corresponding
to the characters $A^s\setminus A$, induces a surjective morphism
\[
\eta\colon \pi_1(\C^{A}\setminus V_{A}, x) \rightarrow \pi_1\big(\C^{A^s}\setminus V_{A^s}, (x,0)\big).
\]
\end{theorem}

\smallskip
Throughout this paper, we say that a morphism between fundamental groups is \emph{canonical}
if it is induced by an inclusion of topological spaces, as in Theorem~\ref{thm:MainFundamentalGroup}.
An interesting question, which we do not address in detail, is whether the canonical morphism in Theorem~\ref{thm:MainFundamentalGroup} is an isomorphism. This can be deduced in
special cases, as in \S\ref{sec:Lines}.

\smallskip
With this result in hand, we return to $A$-hypergeometric
monodromy. The article~\cite{Beu16} describes an algorithm to
compute the monodromy group $\Mon_x(A, \beta)$ of an
$A$-hypergeometric system with parameter vector $\beta$ at the base
point $x\in \C^A$. However, this method is only applicable if the
collection $A$ satisfies some assumptions. The milder of these assumptions, that $A$ admits a 
so-called \emph{Mellin--Barnes basis}, fails already for relatively small collections $A$.
One can deduce from~\cite{For15, FJ15} that the set of all collections 
$A$ which admit a Mellin--Barnes basis is a \emph{semi-ideal} (or, \emph{downward closed set})
in the poset lattice of all collections with Newton polytope $N$. 
In other words, a suitable subcollection of $A$ can admit a Mellin--Barnes basis,
even if $A$ does not. Our main result supplements this semi-ideal property
by describing conditions under which the monodromy group is invariant under the 
actions of deleting (or adding) characters from $A$.

\begin{theorem}
\label{thm:MainMonodromy}
Assume that $A$ has an interior point.
If the parameter $\beta$ is sufficiently generic (i.e., \emph{nonresonant}; see Definition~\ref{def:nonresonant}), then 
\[
\Mon_{x}(A, \beta) \simeq \Mon_{(x,0)}(A^s, \beta).
\]
\end{theorem}

The genericity assumption on the parameter $\beta$ cannot be
removed, since it is known to characterize $A$-hypergeometric systems
with irreducible monodromy representations \cite{Beu11, SW12}.
Without the genericity assumption, it might even be that the dimensions of the solution spaces
differ \cite{CDD99, SST00}.

Finally, we remark that we prove stronger versions of Theorems~\ref{thm:MainFundamentalGroup} 
and \ref{thm:MainMonodromy}, which apply also in the situation
when $A$ has no interior points. The necessary definitions are, naturally, more technical, and we
have saved the details for Definition~\ref{def:saturations} and Theorems~\ref{thm:MainFundamentalGroup2} and \ref{thm:MainMonodromy2}.

\subsection*{Outline.} 
Section~\ref{sec:Notation} sets notation and reviews necessary background.
Section~\ref{sec:FundamentalGroups} begins our study of fundamental
groups. Section~\ref{sec:Saturations} links these results to the
combinatorics of the collection $A$, and contains the proof of
Theorem~\ref{thm:MainFundamentalGroup}. We turn to $A$-hypergeometric
monodromy and prove Theorem~\ref{thm:MainMonodromy} in
Section~\ref{sec:Monodromy}. Finally, we apply these results and
Beuker's method to compute the monodromy groups for $A$-hypergeometric
systems associated to monomial curves in Section~\ref{sec:Lines}.

\subsection*{Acknowledgements}
We cordially thank Michael L{\"o}nne and Frits Beukers for enlightening discussions.
JF gratefully acknowledges the support of the Netherlands Organization for Scientific Research (NWO),
grant TOP1EW.15.313.
\section{Preliminaries}
\label{sec:Notation}

We use, with slight adjustments, 
the notation of \cite[Chapter~10]{GKZ94}.

Throughout this article,  $A = \{\alpha_1, \dots, \alpha_k\} \subset \Z^{1+n}$ denotes
a quasi-homogeneous collection of algebraic characters of the torus
$(\C^*)^{1+n}$, and $\beta\in \C^{1+n}$ is a parameter vector. 
We often write $A$ as a matrix whose columns are the characters $\alpha_i$.
We
denote by $X_A$ the toric variety associated to $A$, with $\Z_A$ its
character lattice, and we let $N = \conv(A) \subset \R \otimes \Z_A$
denote the corresponding Newton polytope. The collection $A$ is said to be \emph{saturated} 
if $A = N \cap \Z_A$. However, we impose no such assumption.
Let $\N_A = \N[A]$ denote the monoid generated by $A$ and the origin,
so that $\Z_A$ is the group completion of $\N_A$. Given an affine lattice $L$, the 
\emph{lattice volume} defined by $L$ is the unique translation
invariant measure on $\R \otimes L$ such that a minimal simplex in $L$ has volume one.

\smallskip
Let $\cF$ denote the face poset lattice of $N$, and let $\cF_{\intt} \subset \cF$
denote the semi-ideal generated by all faces $\Gamma \preccurlyeq N$ that 
contain a relative interior point in $A$. 
Given $\Gamma \in \cF$, let $\Gamma_\R$ denote the linear span of $\Gamma$ and the origin.
That is, $\Gamma_\R$ is the linear span of the cone generated by the face $\Gamma \preccurlyeq N$.
(Some authors prefer to consider the cone over $N$ rather than $N$;
the face lattices of the two coincide except for the apex of the
cone.) 

Define the \emph{index} 
\[
i(A, \Gamma) = \big[\Z_A\cap \Gamma_\R \, : \, \Z_{A \cap \Gamma}\big].
\]
Set $\Z_A/\Gamma = \Z_A/(\Z_A \cap \Gamma_\R)$, 
and consider the admissible semigroup
\begin{equation}
\label{eqn:SubdiagramVolume}
\N_A/\Gamma = (\Z_{A \cap \Gamma} + \N_A) \big/ (\Z_A \cap \Gamma_\R) \subset \Z_A/\Gamma.
\end{equation}
Let $v(A, \Gamma)$ denote the \emph{subdiagram volume} of $\N_A/\Gamma$.
That is, $v(A, \Gamma)$ denotes the lattice volume of the set difference between the convex hulls of $\N_A/\Gamma$ and $(\N_A/\Gamma)^* = (\N_A/\Gamma)\setminus \{0\}$ in
$ \R\otimes \Z_A/\Gamma$.
By convention, the subdiagram volume of the trivial semigroup is one.

\smallskip
Let $z=(z_0,\dots,z_n)$ be coordinates on the torus $(\C^*)^{1+n}$. Given $\alpha_i \in \Z^{1+n}$,
the associated character is the monomial $z \mapsto z^{\alpha_i}$. We
use $\C^A$ to denote the space of polynomials
\[
  \C^A=\big\{ y_1z^{\alpha_1}+\cdots +y_kz^{\alpha_k} \mid y_1,\dots, y_k
  \in \C\big\}.
\]
That is, $\C^A$ is complex vector space of dimension $k$ with
coordinates $y = (y_1,\dots,y_k)$.

The collection $A$ defines a projective toric variety $X_A \subset \P(\C^k)$.
Let $\P(\C^A)$ denote the dual space of $\P(\C^k)$. 
The \emph{$A$-discriminantal variety} is the projectively dual
\begin{equation}
\label{eqn:ADiscriminant}
\widecheck{X}_A \subset \P(\C^A).
\end{equation}
The collection $A$ is said to be \emph{nondefective} if the $A$-discriminant is a hypersurface, in which case
we denote by $D_A$ its defining homogeneous polynomial (unique up to sign, if the coefficients are required to be relatively prime integers). If $A$ is defective, then $D_A = 1$.
The affine cone $\widecheck{X}_A \subset \C^A$
is the closure of the rational locus of all polynomials
$f \in (\C^*)^A$ which has a singularity in $(\C^*)^{1+n}$.

Following~\cite{GKZ94}, the \emph{principal $A$-determinant} is defined as the toric resultant
\begin{equation}
\label{eqn:PrincipalADeterminant}
E_A(f)= R_A\Big( z_0 \frac{\partial f}{\partial z_0}, \dots, z_n \frac{\partial f}{\partial z_n}\Big). 
\end{equation}
We make use of the formula \cite[Ch.~10, Thm.~1.2]{GKZ94}, up to a nonzero constant,
\begin{equation}
\label{eqn:PrincipalADeterminant2}
E_A(f) = \prod_{\Gamma \in \cF} D_{A \cap \Gamma}(x)^{m(A, \Gamma)} ,
\end{equation}
where the multiplicities $m(A, \Gamma)$ are given by
\begin{equation}
\label{eqn:multiplicities}
m(A, \Gamma) = i(A, \Gamma)\, v(A, \Gamma).
\end{equation}
Notice that $m(A, N) = 1$ and that  $m(A, \Gamma) \geq 1$  for all $\Gamma \in \cF$. As fundamental groups are topological rather than algebraic, we often replace the principal $A$-determinant with the reduced polynomial
\begin{equation}
\label{eqn:reduced}
\widehat E_A(x) =  \prod_{\Gamma \in \cF} D_{A \cap \Gamma}(x).
\end{equation}

\begin{lemma}
\label{lem:PADRestriction}
Assume that $\alpha_i \in A$ is not a vertex of $N$. 
Set $Y_i = \{\,x \in \C^A\, | \, x_i = 0\,\}$ and $A_i = A\, \setminus \{\alpha_i\}$.
Then, each irreducible component of\/ $V_A \cap Y_i$ is contained in $V_{A_i}$.
\end{lemma}

\begin{proof}
Since $\alpha_i$ is not a vertex of $N$, the hyperplane $Y_i$ is not contained in $V_A$.
It follows that the restriction of $E_A$ to $Y_i$ is nontrivial. The statement then follows
from \eqref{eqn:PrincipalADeterminant}.
\end{proof}
In particular, there is an identity of sets $V_A \cap Y_i = V_{A_i}$.
However, the irreducible components of $V_A \cap Y_i$
need not appear with the same multiplicity in $V_{A_i}$. In practice, tracing how the multiplicities of irreducible components
in $V_A$ change when we restrict to coordinate hyperplanes is a central part of our investigation.

\smallskip
Consider a (regular) triangulation $T$ of the Newton polytope $N$, with vertices in $A$.
We express $T$ as the set of full-dimensional cells $\sigma$.
Consider the characteristic function $\varphi_T\colon A \rightarrow \Z$
defined by
\[
\varphi_T(\alpha) = \sum_{\sigma \in T \, | \, \alpha \in \operatorname{vert}(\sigma)} \, \operatorname{vol}_{\Z_A}(\sigma).
\]
That is, $\varphi_T(\alpha)$ is the sum of the lattice volumes of all simplices in $T$ containing $\alpha$
as a vertex.
The \emph{secondary polytope} $\Sigma_A$ is defined as the convex hull of the vectors 
\[
\varphi_T(A)= (\varphi_T(\alpha_1), \dots, \varphi_T(\alpha_k)) \in \Z^k
\]
as $T$ ranges over all (regular) triangulations
of $A$ \cite[Ch.~7]{GKZ94}. The secondary polytope coincides with the Newton polytope
of the principal $A$-determinant $E_A$ \cite[Ch.~10, Thm.~1.4]{GKZ94}. We make the following remark,
where we use coordinates $u$ for $\Z^k$.

\begin{lemma}
\label{lem:SecondaryPolytopeRestriction}
Assume that $\alpha_i \in A$ is not a vertex of $N$.  If\/ $\Z_A = \Z_{A_i}$, then the secondary polytope $\Sigma_{A_i}$ coincides with the facet of\/ $\Sigma_A$ contained in the hyperplane $u_i  = 0$.
\end{lemma}

\begin{proof}
Since $\Z_{A_i} = \Z_A$, they induce the same lattice volume. Hence, it suffices to note that a regular triangulation $T$ of $N$, with vertices in $A$, is such that $u = \varphi_T(A)$ has $u_i = 0$ 
if and only if $\alpha_i$ is not a vertex of any simplex in $T$.
\end{proof}

\section{Fundamental Groups}
\label{sec:FundamentalGroups}

Throughout this section, let us consider a general polynomial $P \in \C[y_1, \dots, y_k]$.
By slight abuse of notation, we denote by $V \subset \C^k$ both the vanishing
locus of $P$ and the set of irreducible components of $V$. 
An irreducible hypersurface in $V$ will be denoted by a capital letter,
and points in $V$ will be denoted by lowercase letters. We use $y$ as coordinates,
and denote the base point of the fundamental group by $x$.

Let $Z \in V$ be an irreducible component, and choose a smooth point $y \in Z$. 
For a generic line $\ell$ passing through $Z$,
and a sufficiently small open neighborhood $U$ of $y$,
the complement $(U\cap\ell) \setminus Z$ is a punctured disc.
Choose an auxiliary point $\hat x \in (U\cap\ell) \setminus Z$, choose a generator
$\hat \gamma$ of $\pi_1\big((U\cap\ell) \setminus Z, \hat x\big) \simeq \Z$,
and choose a path $\rho$ from $x$ to $\hat x$ in $\C^k \,\setminus V$.
Then, the \emph{generator-of-the-monodromy} (gom) of $\pi_1(\C^k\, \setminus V, x)$,
around $Z$ and determined by the above choices, is the path
\[
\gamma = \rho^{-1} \circ \hat \gamma \circ \rho.
\]
The nomenclature is self-explanatory; it is well known that the set of goms around all irreducible components $Z \in V$ generates the fundamental group $\pi_1(\C^k\,\setminus V, x)$. The following lemma
is also classical.

\begin{lemma}[See, e.g., {\cite[Lem.~2.1]{Bes01}}]
\label{lem:BessisLemma}
Let $V_1$ and $V_2$ be two disjoint families of irreducible hypersurfaces in $\C^k$, 
and choose $x \in \C^k\,\setminus (V_1 \cup V_2)$. 
Then, the canonical homomorphism
\[
\eta\colon \pi_1\big(\C^k\,\setminus (V_1\cup V_2), x\big) \rightarrow \pi_1\big(\C^k\,\setminus V_1, x\big)
\]
is surjective and, more precisely:
\begin{enumerate}
\item Each gom of $\pi_1\big(\C^k\,\setminus V_1, x\big)$
lifts to a gom of $\pi_1\big(\C^k\,\setminus (V_1\cup V_2), x\big)$.
\label{item:LemmaPart1}
\item The kernel of $\eta$ is generated by the goms around components of $V_2$.
\label{item:LemmaPart2} \qed
\end{enumerate}
\end{lemma}

Let $P \in \C[y_1, \dots, y_k]$ with vanishing locus $V$. We are interested in the intersection $V \cap Y_i$
where $Y_i = \{\,y \in \C^k \, | \, y_i  = 0\,\}$. To simplify the presentation, we assume that $i = k$.
In the following proposition, we work against an auxiliary variable, which we can assume to be $y_1$.
We use $y = (y_1, \bar y, y_k)$, where $\bar y = (y_2, \dots, y_{k-1})$, as coordinates
of $\C^k$. Let
\[
\disc_j(P) =  \res_{j}\Big(P, \frac{\partial P}{\partial y_j}\Big)
\]
denote the \emph{discriminant} of $P$ with respect to $y_j$. That is, $\disc_j(P)$ is the resultant
of $P$ the derivative $\partial P/\partial y_j$ with respect to the variable $y_j$.
Notice that $\disc_j(P)$ does not depend on $y_j$.

\begin{proposition}
\label{Prop:FundamentalGroupMain}
Let $V$ be a hypersurface defined by a polynomial $P \in \C[y_1, \bar y, y_k]$. Let 
\[
K(\bar y, y_k) = \disc_{1}(P)(\bar y, y_k)
\]
denote the discriminant of $P$ with respect to the auxiliary variable $y_1$. If
\begin{enumerate}
\item all common factors of $K$ and $P$ belong to $\C[\bar y]$, and  \label{item:Theorem1}
\item $K$ restricted to $Y_k$ is nontrivial \label{item:Theorem2}
\end{enumerate}
then, for a base point $x = (x_1, \bar x, 0) \in Y_k \setminus V$, the canonical morphism
\[
\eta \colon \pi_1\big(Y_k \setminus (Y_k \cap V), x\big) \rightarrow \pi_1(\C^k\,\setminus V, x)
\]
is surjective.
\end{proposition}

\begin{proof}
Let $W$ be the set of irreducible components of the hypersurface $K(\bar y, y) = 0$ in $\C^k$.
By the assumptions  \eqref{item:Theorem1} and \eqref{item:Theorem2} we can choose a base point $x \in (Y_1\cap Y_k) \setminus (V \cup W)$.
Write $V = V_K \cup V_{\widehat K}$, where $V_K  =V \cap W$
and $V_{\widehat K} = V \setminus W$. (Here, we view $V$, $V_K$, and $V_{\widehat K}$ as sets
of irreducible hypersurfaces in $\C^k$.)

By the assumption \eqref{item:Theorem1}, 
each element of $V_K$ is the vanishing locus of a polynomial from 
$\C[\bar y] \subset \C[y_1, \bar y, y_k]$ and, hence, 
$\C^k\,\setminus V_K$ is a trivial bundle over
$Y_k\setminus (Y_k\cap V_K)$ with fibers isomorphic to $\C$.
It follows that the morphism $\eta'$ in the commutative diagram
\begin{center}
\begin{tikzcd}
& \pi_1 \big(Y_k \setminus (Y_k \cap V), x\big) \arrow[r, "\eta"]\arrow[d]
& \pi_1 (\C^k\, \setminus V, x)\arrow[d, "\theta"] \\
& \pi_1 \big(Y_k \setminus (Y_k \cap V_K), x\big) \arrow[r, "\eta'"]
& \pi_1 (\C^k\, \setminus V_K, x)
\end{tikzcd}
\end{center}
is an isomorphism. Let $\gamma \in \pi_1(\C^k\,\setminus V, x)$.
Since $\eta'$ is an isomorphism, there is an element $\gamma' \in \pi_1(Y_k\setminus (Y_k\cap V_K), x)$
such that $\theta(\gamma)\eta'(\gamma') = 0$, and it follows from Lemma~\ref{lem:BessisLemma} part \eqref{item:LemmaPart1} that $\gamma'$ lifts to $\pi_1\big(Y_k \setminus (Y_k\cap V), x\big)$.
Hence, it suffices to show that each $\gamma \in  \pi_1(\C^k\,\setminus V, x)$
with $\theta(\gamma) = 0$ belongs to the image of $\eta$.

Assume that $\gamma \in  \pi_1(\C^k\,\setminus V, x)$
belongs to the kernel of the 
morphism $\theta$. It follows from Lemma~\ref{lem:BessisLemma} part \eqref{item:LemmaPart2}
that $\gamma$ belongs to the subgroup generated by the goms
around $V_{\widehat K}$. Hence, there is no loss of generality in assuming that $\gamma$
is a gom around $V_{\widehat K}$.

The final step is analogous to the argument of \cite[Thm.~2.5]{Bes01}.
Let $\delta$ be the degree of $P$ in the variable $y_1$.
We obtain a trivial fiber bundle 
\[
\C^k\,\setminus (V \cup W) \rightarrow Y_1 \setminus W,
\]
whose fibers are complex lines with $\delta$ points removed, 
and we obtain the long exact sequence of homotopy groups 
\begin{center}
\begin{tikzcd}
\dots \arrow[r]  
& \pi_1\big(L \setminus (L\cap V), x\big) \arrow[r, "\tau"] 
& \pi_1\big(\C^k\,\setminus (V\cup W), x\big) \arrow[r, "\tau'"] 
& \pi_1\big(Y_1 \setminus W, x\big) \arrow[r]
& 0.
\end{tikzcd}
\end{center}
Any gom $\gamma$ of $V_{\widehat K}$ in $\pi_1\big( \C^k\,\setminus V, x\big)$
lifts by Lemma~\ref{lem:BessisLemma} part \eqref{item:LemmaPart1} to a gom of $V_{\widehat K}$ in $\pi_1\big( \C^k\,\setminus (V\cup W), x\big)$,
which belongs to the kernel of the morphism $\tau'$. Hence, $\gamma$ lies in the image of the
morphism $\tau$.
But \eqref{item:Theorem2} implies that we can choose the fiber $L$  inside the plane $Y_k$
and, hence, $\tau$ is simply the morphism $\eta$ restricted to
$\pi_1\big(L\setminus (L \cap V), x \big)$. The result follows.
\end{proof}

We now translate the geometric conditions of Proposition~\ref{Prop:FundamentalGroupMain}
in terms of combinatorial conditions on the collection $A$. 
Recall that restricting to the hyperplane $Y_k$ corresponds to deleting the point $\alpha_k \in A$. 
The auxiliary variable $y_1$ corresponds to an auxiliary point $\alpha_1 \in A$. Before stating
this result, we need a combinatorial definition.

\begin{definition}
  \label{def:latticeRedundant}
Let $\alpha_i \in A$ and set $A_i = A \,\setminus\{\alpha_i\}$.
We say that $\alpha_i$ is \emph{lattice redundant}
if all face lattices of $A$ and $A_i$ coincide. That is, if for each face $\Gamma \preccurlyeq N$ we have 
that $\Z_{A \cap \Gamma} = \Z_{A_i \cap \Gamma}$.
\end{definition}

\begin{proposition}
\label{pro:AZariski}
Let $\alpha_k \in A$ be lattice redundant. Let $\alpha_1$ be an auxiliary point, 
contained in a minimal face $\Gamma_1 \preccurlyeq N$. Assume, in addition, that
\begin{enumerate}
\item $\alpha_k$ is contained in the closure of\/ $\Gamma_1$, and 
\item if $\alpha_1 \in \Gamma_2$,
then either $A\cap \Gamma_2$ is defective or $m(A, \Gamma_2) = m(A_k, \Gamma_2)$.
\end{enumerate}
Then, the canonical morphism
\[
\eta\colon \pi_1\big(\C^{A_k} \setminus E_{A_k}, (x_1, \bar x)\big) \rightarrow \pi_1\big(\C^A \setminus E_A, (x_1,  \bar x, 0)\big)
\]
is surjective.
\end{proposition}

\begin{proof}
Recall that $D_{A\cap \Gamma}$ denotes the $A\cap\Gamma$-discriminant, which
appear as a factor of the principal $A$-determinant $E_A(f)$ with multiplicity
$m(A, \Gamma)$, see \eqref{eqn:ADiscriminant} and \eqref{eqn:PrincipalADeterminant2}.

We first claim that for each face $\Gamma \preccurlyeq N$ there is a polynomial $Q_\Gamma \in \C[\bar y]$
such that 
\[
D_{A\cap \Gamma}(y_1, \bar y, 0) = Q_\Gamma(\bar y) D_{A_k \cap \Gamma}(y_1,\bar y).
\]
A priori, we know, by Lemma~\ref{lem:PADRestriction}, that $Q_\Gamma \in \C[y_1, \bar y]$ is a product of $A\cap F$-discriminants of 
faces $F \preccurlyeq \Gamma$
(possibly including the face $F = \Gamma$).
Let $d(F)$ denote the multiplicity of the $A\cap F$-discriminant
in the product of all coefficients $Q_\Gamma$ as $\Gamma$ ranges over all faces of $N$.
It follows that, up to a constant,
\[
E_A(y_1, \bar y, 0) = \prod_{\Gamma \preccurlyeq N} D_{A_k \cap \Gamma}^{m(A,\Gamma) + d(\Gamma)}(y_1, \bar y).
\]
Since $\alpha_k$ is lattice redundant, the Newton polytope of $E_A(y_1, \bar y, 0) \in \C[y_1, \bar y]$
coincides with the secondary polytope $\Sigma_{A_k}$ by Lemma~\ref{lem:SecondaryPolytopeRestriction}.
Hence, the degree in $y_1$ of $E_A(y_1, \bar y, 0)$ and $E_{A_k}(y_1, \bar y)$ coincide. 
Counting said degree, we obtain
\[
\sum_{\Gamma \preccurlyeq N}(m(A, \Gamma) + d(\Gamma)) \deg_1(D_{A_k \cap \Gamma}) = \sum_{\Gamma \preccurlyeq N} m(A_k, \Gamma) \deg_1(D_{A_k \cap \Gamma}).
\]
If $A\cap \Gamma$ is defective or if $\alpha_1 \notin \Gamma$, then $\deg_1(D_{A_k \cap \Gamma}) = 0$.
For the remaining faces, we are under the assumption that $m(A, \Gamma) = m(A_k, \Gamma)$.
It follows that 
\[
\sum d(\Gamma) \deg_1(D_{A_k \cap \Gamma}) =0,
\]
where the sum is taken over all faces $\Gamma$ containing $y_1$ such that $A \cap \Gamma$ is nondefective.
As the degree $\deg_1(D_{A_k \cap \Gamma})$ is positive for such faces,
we conclude that $d(\Gamma) = 0$ for each such face. It follows that no coefficient $Q_\Gamma$
contains a factor which depends on $y_1$ and, hence, $Q_\Gamma \in \C[\bar y]$ as claimed.

That $Q_\Gamma(\bar y)$ does not depend on $y_1$ implies that
\[
\disc_1(\widehat E_A)(\bar y, y_k)\Big|_{Y_k}
= \disc_1(\widehat E_A)(\bar y, 0)
= \big(\disc_1(\widehat E_{A_k})(\bar y)\big)
\prod_{\Gamma \preccurlyeq N} Q_\Gamma(\bar y),
\]
where $\widehat E_A$ and $\widehat E_{A_k}$ denotes the reduced principal $A$-determinants
from \eqref{eqn:reduced}.
Since $\widehat E_{A_k}$ is reduced, the first factor in the right hand side is
nontrivial. It follows that
\[
K(\bar y, y_k) = \disc_1(\widehat E_A)(\bar y, y_k)
\]
is nontrivial when restricted to $Y_k$. The common factors of $K(\bar y, y_k)$ and
$\widehat E_A(y_1, \bar y, y_k)$ are the $(A\cap \Gamma)$-discriminants for
faces $\Gamma$ not containing $y_1$. By assumption, any face containing
$y_k$ also contains $y_1$ and, hence, said common factors
belongs to $\C[\bar y]$. That is, all assumptions
of Proposition~\ref{Prop:FundamentalGroupMain} are fulfilled.
\end{proof}

\section{Saturations}
\label{sec:Saturations}

Recall that $A$ is saturated if $A = N \cap \Z_A$.
We now give the precise definition of the face saturations alluded to in the introduction.
Recall that $\cF$ denotes the face poset lattice of $N$, and that $\cF_{\intt} \subset \cF$
denote the semi-ideal of $\cF$ generated by all faces $\Gamma \preccurlyeq N$ that 
contain a relative interior point in $A$. 

\begin{definition}
  \label{def:saturations}
The  \emph{face saturation} $A^s$ and the
\emph{partial face saturation} $A^p$ of the collection $A$ are given by 
\[
A^s =\bigcup_{\Gamma \,\in \,\cF} \Gamma^\circ\cap \Z_{A\cap \Gamma}
\quad \text{and} \quad
A^p =\bigcup_{\Gamma \,\in \,\cF_{\intt}} \Gamma^\circ\cap \Z_{A\cap \Gamma},
\]
where $\Gamma^\circ$ denotes the relative interior of the face $\Gamma$.
\end{definition}

That is, $A^s$ is obtained from $A$ by adjoining, for each face $\Gamma \preccurlyeq N$,
all points of the lattice $\Z_{A \cap \Gamma}$ which are contained in the
relative interior of $\Gamma$. We say that we \emph{saturate each face of $N$}.
The partial face saturation differs from the
face saturation in that we only saturate the faces which belong to the closure of
a face that already contains a relative interior point.

\begin{example}
\label{ex:saturation}
\begin{figure}[t]
\includegraphics[width=100mm]{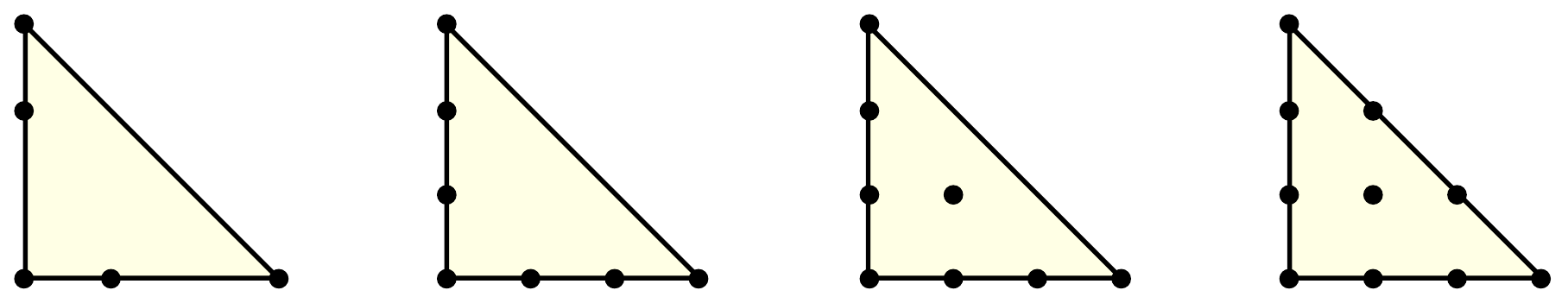}
\caption{The collections $A, A^p, A^s,$ and $N\cap \Z_A$ from Example~\ref{ex:saturation}.}
\label{fig:Saturations}
\end{figure}
Consider the two dimensional collection
\[
A = \left[\begin{array}{ccccc} 
1 & 1 & 1 & 1 & 1\\
0 & 3 & 0 & 1 & 0\\
0 & 0 & 3 & 0 & 2
\end{array}\right].
\]
Notice that $\Z_A = \Z^3$, see Figure~\ref{fig:Saturations}. 
The relative face saturation $A^p$ is obtained by saturating the two
edges of $N$ which contain a relative interior point. That face saturation
$A^s$ is obtained from $A^p$ by also saturating the full-dimensional 
face relative $\Z_A$. Note that we only add the interior point
$(1,1,1)$ at this step. The saturation $N\cap \Z_A$ consists
of all integer points in the Newton polytope.
\end{example}

\begin{proposition}
\label{prop:InductionStep}
Let $\alpha_k \in A$ be lattice redundant, and assume that there is
a face $\Gamma \in \cF_{\intt}$, containing $\alpha_k$,
which has a relative interior point $\alpha_1$ distinct from $\alpha_k$.
Then, the canonical morphism
\[
\eta\colon \pi_1\big(\C^{A_k} \setminus E_{A_k}, (x_1, \bar x)\big) \rightarrow \pi_1\big(\C^A \setminus E_A, (x_1,  \bar x, 0)\big)
\]
is surjective.
\end{proposition}

\begin{proof}
We only need to check that the conditions of Proposition~\ref{pro:AZariski} are fulfilled,
with $\alpha_1$ as the auxiliary point.
Recall the expression for the multiplicities $m(A, \Gamma)$ from \eqref{eqn:multiplicities}.
Since $\alpha_k$ is lattice redundant, we have that $i(A, \Gamma) = i(A_k, \Gamma)$
for all faces $\Gamma \in \cF$. It remains to be checked that the subdiagram volume 
is invariant for any face $\Gamma$ containing $\alpha_1$. This follows a fortiori
from the equality of semigroups $\N_{A_k}/\Gamma  = \N_A/\Gamma$
(cf.~\eqref{eqn:SubdiagramVolume}). Indeed, that
$\Z_A\cap \Gamma_\R = \Z_{A_k} \cap \Gamma_\R$
holds because $\alpha_k$ is lattice redundant. The only generator of $\N_A$ not contained
in $\N_{A_k}$ is $\alpha_k$, but $\alpha_k \in \Z_{A \cap \Gamma} = \Z_{A_k\cap \Gamma}$,
where the last equality holds because of lattice redundancy.
Consequently, $\Z_{A \cap \Gamma} + \N_A = \Z_{A_k \cap \Gamma} + \N_{A_k}$,
from which we deduce that $\N_{A_k}/\Gamma  = \N_A/\Gamma$. Since the semigroups coincide,
so do their subdiagram volumes.
\end{proof}

We are now ready to prove the stronger version of Theorem~\ref{thm:MainFundamentalGroup}.

\begin{theorem}
\label{thm:MainFundamentalGroup2}
Assume that $A$ has an interior point. Then,
the inclusion $\C^A \rightarrow \C^{A^p}$ given by appending zeros to $x$ for the coordinates corresponding
to the characters $A^p\setminus A$, induces a surjective morphism
\[
\eta\colon \pi_1(\C^{A}\setminus V_{A}, x) \rightarrow \pi_1\big(\C^{A^p}\setminus V_{A^p}, (x,0)\big).
\]
\end{theorem}

\begin{proof}
The proof is a simple induction using Proposition~\ref{prop:InductionStep}.
\end{proof}

\begin{proof}[Proof of Theorem~\ref{thm:MainFundamentalGroup}]
If $A$ has a relative interior point, then $A^p = A^s$.
Hence, the statement follows from Theorem~\ref{thm:MainFundamentalGroup2}.
\end{proof}

For the remainder of this section, we discuss the question of whether,
in general, the partial face saturation $A^p$ can be 
replaced by the face saturation $A^s$ in Theorem~\ref{thm:MainFundamentalGroup2}. 
The main remark is that Proposition~\ref{pro:AZariski} 
is much stronger than what is needed to prove Proposition~\ref{prop:InductionStep}. 

\begin{proposition}
\label{prop:Dim2FaceSaturations}
If the collection $A$ has dimension at most two, then Theorem~\ref{thm:MainFundamentalGroup}
still holds if the partial face saturation $A^p$ is replaced by the face saturation $A^s$.
\end{proposition}

\begin{proof}
If $A$ has dimension one, then $A^p = A^s$. Suppose that $A$ has dimension two.
There is no loss of generality in assuming that $A =  A^p$ and that
$A$ has no relative interior points. 
Let $\alpha_1$ be a vertex of $N$ and, for $i = 2,3$, let $\Gamma_i = \conv(\{\alpha_1, \alpha_i\})$ denote the edges of $N$ incident to $\alpha_1$.
Let $\ell_i$ denote the lattice length of $\Gamma_i$ relative $\Z_{A \cap \Gamma_i}$,
so that $\alpha_1 + (\alpha_i - \alpha_1)/\ell_i$ is the closest point in to $\alpha_1$ in $\Z_{A \cap \Gamma_i}$.
Our first claim is that either $\Z_A$ has no points in $N^\circ$,
or we can find a vertex $\alpha_1$ of $N$ such that
\[
\alpha' = \alpha_1 + (\alpha_2 - \alpha_1)/\ell_2 + (\alpha_3 - \alpha_1)/\ell_3 \in N^\circ.
\]
The proof is elementary planar geometry, and is left to the reader.

\smallskip
If $\Z_A$ has no points in $N^\circ$ then $A^s = A^p$ and there is nothing to prove.
If  $\Z_A$ has a point in $N^\circ$, then $\alpha' \in N^\circ$.
Let $A' = A \cup \{\alpha'\}$. By construction, $\alpha'$ is lattice redundant in $\Z_{A'}$.
We claim that $\alpha_1$ can act as an auxiliary point, fulfilling the conditions of
Proposition~\ref{pro:AZariski}. 
It suffices to check that the subdiagram volumes of the faces $\{\alpha_1\}, \Gamma_2,$ and $\Gamma_3$ coincide in $A$ and in $A'$. 
For the face $\{\alpha_1\}$, this follows from that $\alpha'$ is not a vertex of 
$N(A' \setminus \{\alpha_1\})$. For the face $\Gamma_3$,
this holds because $\alpha'$ is equivalent to $(\alpha_2-\alpha_1)/\ell_2$ modulo  $\Z_A \cap (\Gamma_3)_\R$, and vice versa for $\Gamma_2$.
Composing the morphisms obtained from 
Proposition~\ref{pro:AZariski} and Theorem~\ref{thm:MainFundamentalGroup},
noting that the partial face saturation of $A'$ is $A^s$,
we conclude the proposition.
\end{proof}

Proposition~\ref{prop:Dim2FaceSaturations} implies that Theorem~\ref{thm:MainFundamentalGroup}
is suboptimal in the sense that we can possibly, in the statement of the theorem, replace $A^p$ with a larger collection and still 
obtain a surjective morphism $\eta$.
That said, Proposition~\ref{pro:AZariski}
is not sufficiently strong to conclude that we can
replace $A^p$ by $A^s$  in general,
as shown in the following example.

\begin{example}
\label{ex:counterexample}
\begin{figure}[t]
\includegraphics[height=30mm]{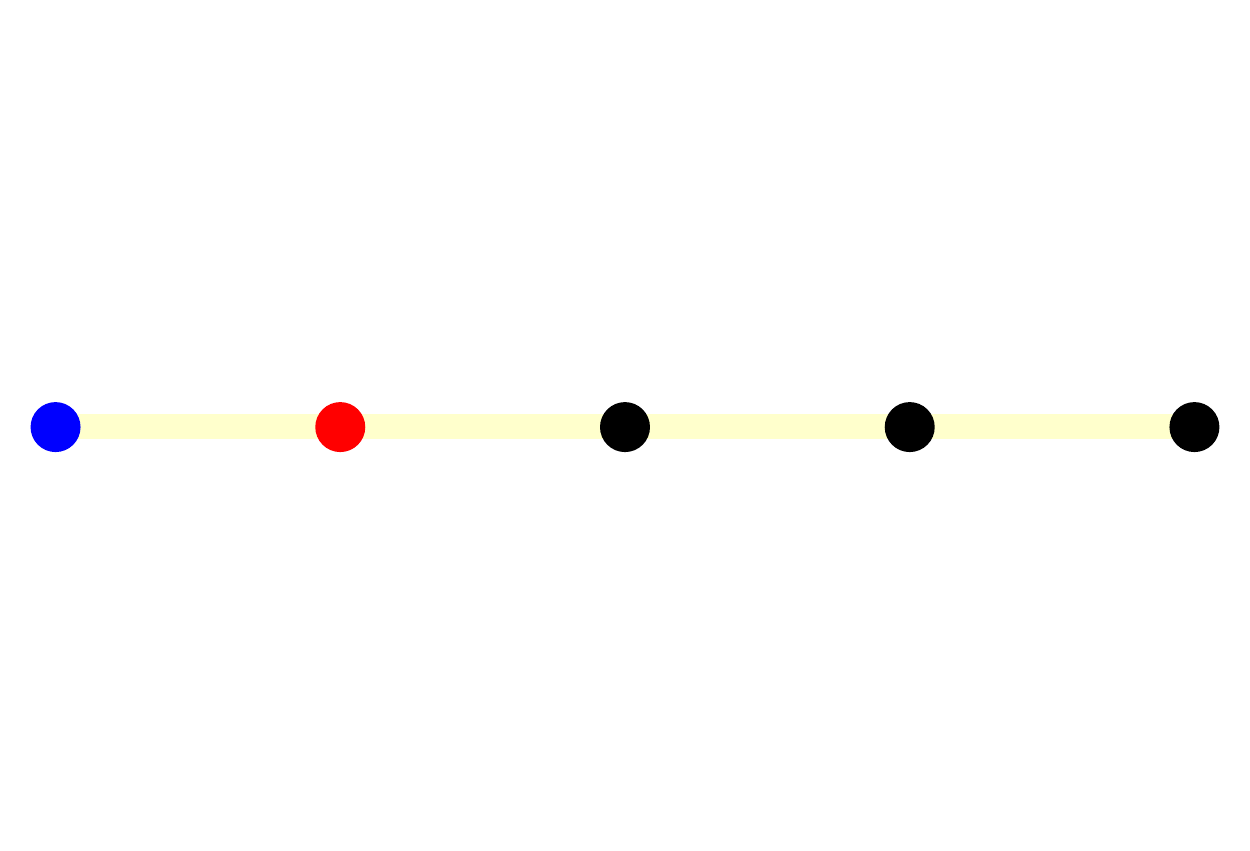}
\hspace{16mm}
\includegraphics[height=30mm]{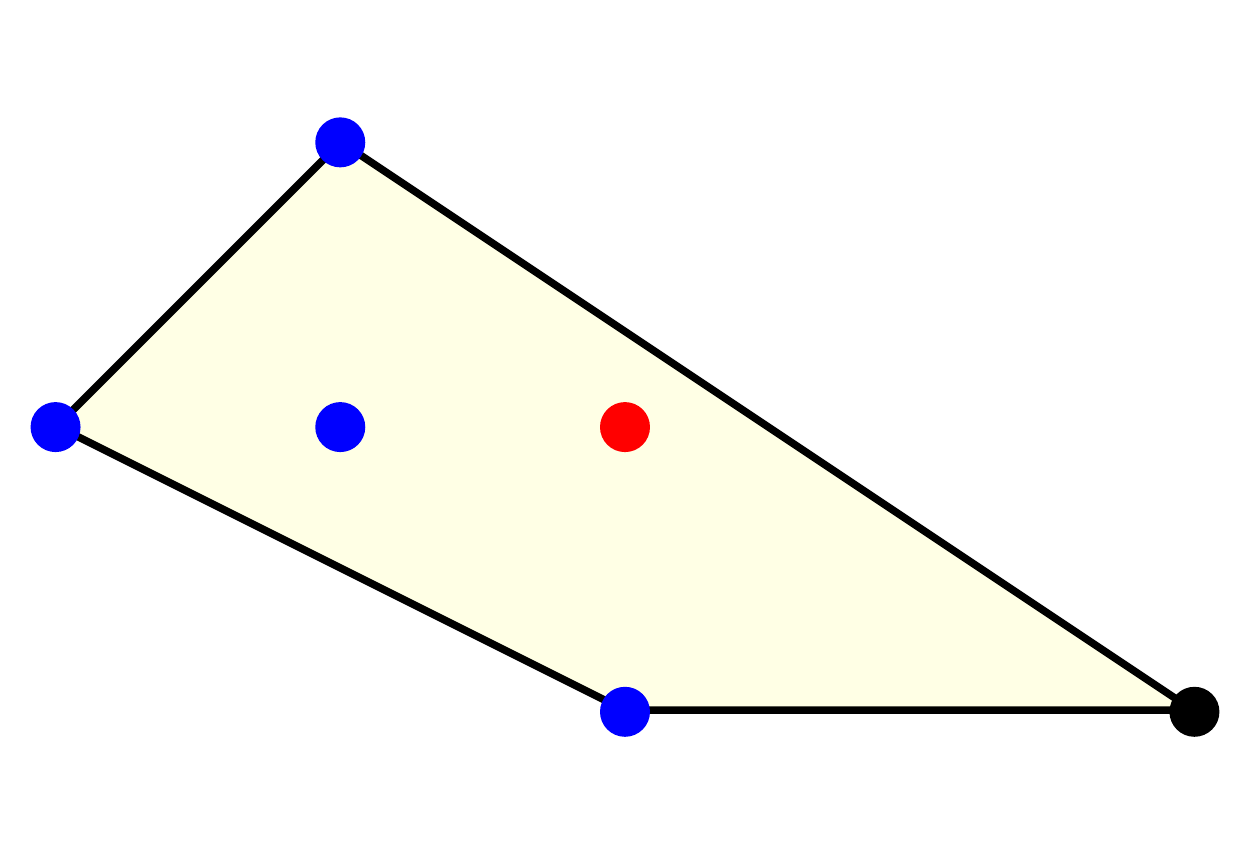}
\caption{The image of $A^s$ under the projections $\pi_1$ and $\pi_2$,
with the images of $\Gamma_1 \cap A$ in blue, the images of $\Gamma_2 \cap A$ in black,
and the images of $\alpha_8$ in red. Notice that $p_i(\alpha_8)$
is a vertex of $\conv(p_i(A^s \setminus \Gamma_i))$ for $i = 1,2$.}
\label{fig:Counterexample}
\end{figure}
Consider the collection $A$, and its face saturation,
\[
A = \left[\begin{array}{ccccccc} 
1 & 1 & 1  & 1 & 1 & 1 & 1\\
0 & 1 & 2 & 1 & 2 & 1 & 0\\
1 & 2 & 0 & 1 & 0 & 0 & 0\\
0 & 0 & 0 & 0 & 2 & 3 & 4
\end{array}\right]
\quad \text{and} \quad 
A^s = A \cup \left[\begin{array}{c} 1 \\ 1 \\ 1 \\ 1 \end{array}\right].
\]
There are two faces of $N$ with relative interior points in $A$, given by
$\Gamma_1 \cap A = \{\alpha_1, \alpha_2, \alpha_3, \alpha_4\}$
and $\Gamma_2 \cap A = \{\alpha_5, \alpha_6, \alpha_7\}$.
All strict faces of $A$ are saturated relative to their
induced lattices. We have that $\Z_A = \Z^4$, and there is a unique point 
$\alpha_8 \in A^s\, \setminus A$. We claim that there is no point in $A$ which can act as the
auxiliary point in Proposition~\ref{pro:AZariski} for $\alpha_8$. 
It suffices to show that we have strict inequalities of subdiagram volumes 
$v(A, \Gamma_i) > v(A^s, \Gamma_i)$ for $i = 1,2$.
Let $p_i$ denote the projection whose kernel is the linear span of $\Gamma_i$.
Then, the strict inequalities of subdiagram volumes is a consequence of
that $p_i(\alpha_8)$ is a vertex
of the polytope $\conv(p_i(A^s \setminus \Gamma_i))$,
see Figure~\ref{fig:Counterexample}.
\end{example}

\section{Monodromy}
\label{sec:Monodromy}

In this section, we interchangeably think of $A$ as a subset
of $\Z^{1+n}$ and as a $(1+n)\times k$ integer matrix whose entries are denoted by $\alpha_{ij}$. 
Let $\cD$ denote the \emph{Weyl algebra} over $\C^A$, that is, the
ring of linear partial differential operators in $k$ variables with polynomial
coefficients over $\C$. This ring is generated as a $\C$-algebra by
$y_1,\dots,y_k,\del_1,\dots,\del_k$ subject to the
relations imposed by the Leibniz rule for derivatives. Here $\del_j$
stands for the partial derivative operator $\del/\del y_j$.

We observe that the polynomial ring $\C[\del_1,\dots,\del_k]=\C[\del]$ is a commutative
subring of the noncommutative ring $\cD$. The following $\C[\del]$-ideal is known
as the \emph{toric ideal} associated to $A$.
\begin{equation}
  \label{eqn:IA}
   I_A = \<\del^{u_+} - \del^{u_-} \mid u \in \Z^k, A\cdot u = 0 \>
   \subset \C[\del], 
\end{equation}
where for $u\in \Z^k$, $(u_+)_i = \max(u_i,0)$ and $(u_-)_i =
\max(-u_i,0)$ for $i=1,\dots,k$, so in particular, $u_+,u_-\in \N^k$
and $u=u_+-u_-$. Note that the monomials $\del^{u_+}$ and $\del^{u_-}$
do not have any variables in common by construction.

We remark that the homogeneity assumption on $A$ means that the ideal $I_A$ is
homogeneous with respect to the standard (total degree) $\N$-grading
on $\C[\del]$.

We can use $A$ to define more differential operators. Set
\begin{equation}
  \label{eqn:Eulers}
  \E_i = \sum_{j=1}^k \alpha_{ij} y_j \del_j, \qquad i =1,\dots, 1+n.
\end{equation}
For $\beta\in \C^{1+n}$ we denote $\E-\beta = \{\E_1-\beta_1,\dots,
\E_{1+n}-\beta_{1+n}\}$. These are called \emph{Euler operators}. Note
that $F(y)$ is annihilated by the operators $\E-\beta$ if and only if
\[
  F(z^{\alpha_1}y_1,\dots,z^{\alpha_k}y_k) = z^{\beta}
  F(y_1,\dots,y_k),
\]
for $z$ in a nonempty open subset of $(\C^*)^{1+n}$ (with respect to
the Euclidean topology on $\C^*$). If $F(y) = \sum \lambda_u y^u$ is a
formal power series, then $F$ is annihilated by $\E-\beta$ if and only
if $A\cdot u = \beta$ whenever $\lambda_u \neq 0$.

\begin{definition}
  \label{def:AHypergeometricSystem}
  Let $A$ be an $(1+n)\times k$ integer matrix and let $\beta \in
  \C^{1+n}$. The \emph{$A$-hypergeometric system with parameter
    $\beta$} is the left $\cD_k$-ideal 
  \[
    H_A(\beta) =
    \cD_k \cdot \big( I_A + \< \E-\beta \>  \big).
  \]
  If $x \in \C^A$ is a nonsingular point of the system of
  partial differential operators $H_A(\beta)$, then the space of germs
  of complex holomorphic solutions of $H_A(\beta)$ at $x$ is denoted
  $\Sol_x(A,\beta)$.
  The dimension of the $\C$-vector space $\Sol_x(A,\beta)$ is the
  \emph{holonomic rank} of $H_A(\beta)$, denoted by $\rank(H_A(\beta))$.
\end{definition}

\begin{remark}
  \label{rmk:RegularHolonomic}
  Since $I_A$ is homogeneous, $\cD/H_A(\beta)$ is a regular holonomic
  $\cD$-module for all $\beta$ \cite{Hot91}.
\end{remark}

In order to state the main result in this section, we need one more
definition.

\begin{definition}
  \label{def:nonresonant}
  A vector $\beta\in \C^{1+n}$ is said to be a \emph{resonant parameter of $A$}
  (or simply \emph{$A$-resonant}) if there is
  $\gamma \in \Z^{1+n}$ such that $\beta - \gamma$ lies in the linear
  span of a codimension one face of $\conv(A)$. Parameters that are
  not resonant are called \emph{nonresonant} (or $A$-nonresonant).
\end{definition}

It is well known \cite{Ado94, GKZ89, MMW05} that if 
$\beta$ is nonresonant then $\rank(H_A(\beta))=\vol(A)$.
Note that the set of resonant parameters of $A$ is the union of
an infinite but locally finite collection of hyperplanes in $\C^{1+n}$, so
that nonresonant parameters are \emph{very generic}.

\begin{theorem}
  \label{thm:SolutionSpaceIsomorphism}
Assume that $A$ and $A_k$ span the same lattice, and
$\conv(A)=\conv(A_k)$ (so in particular $\vol(A)=\vol(A_k)$, and
$A$-nonresonance coincides with $A_k$-nonresonance).
Let $\beta$ be nonresonant.
If $x$ is a generic nonsingular point of $H_{A_k}(\beta)$, then $(x,0)$ is a
nonsingular point of $H_A(\beta)$ and the morphism
\[
\Psi\colon \Sol_{(x,0)}(A, \beta) \rightarrow \Sol_x(A_k, \beta),
\]
given by $\Psi(F)(y_1,\dots,y_k) = F(y_1,\dots,y_{k-1},0)$, is an isomorphism.
\end{theorem}

Note that since $\vol(A)=\vol(A_k)$ and $\beta$ is nonresonant (for
both $A$ and $A_k$), we have
$\rank(H_A(\beta))=\rank(H_{A_k}(\beta))$, so that the solution spaces
we are interested in have the same vector space dimension. The content
of Theorem~\ref{thm:SolutionSpaceIsomorphism} is not that the two
vector spaces are isomorphic (which is obvious from comparing dimensions), but that evaluating $y_k\mapsto 0$
gives an isomorphism. A proof of the following auxiliary result can be found in~\cite[Lemma~7.11]{DMM10}.

\begin{lemma}
  \label{lemma:derivativeIso}
  If $\beta$ is $A$-nonresonant then for any $\gamma \in \N^k$, right
  multiplication by $\del^\gamma$ induces a $\cD$-module isomorphism
  $\cD/H_A(\beta-A\cdot \gamma) \to \cD/H_A(\beta)$. The induced linear
  transformation of solution spaces
  $\Sol_x(A,\beta) \to \Sol_x(A,\beta-A\cdot \gamma)$ is given by
  differentiation. We denote the inverse of this linear transformation
  by $\del^{-\gamma}$.\qed
\end{lemma}

\begin{remark}
  \label{rmk:antiderivativesAreWellDefined}
  In the previous statement, the precise form of the inverse linear
  transformation between solution spaces depends on $\beta$, and
  denoting it by $\del^{-\gamma}$ is an abuse of notation. This is
  justified because whenever we write $\del^{-\gamma}F$, the
  parameter of the hypergeometric function $F$ will be
  understood. 
   Since $\del_i$ and $\del_j$ commute for $i\neq j$,
  $\del_i$ and $\del_j^{-1}$ also commute. Identities such as
  $\del_i^{-1}\del_i^{-1} = \del_i^{-2}$ hold as well, since both are
  inverses of $\del_i^2$. It follows that, as long as
   $\beta$ is nonresonant, $\del^u$ is well-defined for $u\in \Z^k$.

  If $A\cdot u = 0$ and $\beta$ is nonresonant, $\del^u$ acts as the
  the identity on $\Sol_x(A,\beta)$: if $F$ is a
  nonzero solution of 
  $H_A(\beta)$, then $\del^{u_-} (\del^{u}F-F) =
  \del^{u_-}F-\del^{u_+}F = 0$. Therefore, $\del^uF-F$ is in the
  kernel of the linear isomorphism $\del^{u_-}$, whence $\del^uF=F$.
\end{remark}

\begin{proof}[Proof of Theorem~\ref{thm:SolutionSpaceIsomorphism}]
  Since $\alpha_k$ is not a vertex of $\conv(A)$, the singular locus
  of $H_A(\beta)$ does not contain the hyperplane $y_k=0$. Thus, if
  $x\in \C^{k-1}$ is a sufficiently generic nonsingular point of
  $H_{A_k}(\beta)$, then $(x,0)$ is a nonsingular point of
  $H_A(\beta)$.

  For $u=(u_1,\dots,u_k)\in \Z^k$, denote
  $\bar{u}=(u_1,\dots,u_{k-1})\in \Z^{k-1}$.
  
  Let $\psi=\psi(y_1,\dots,y_{k-1}) \in \Sol_x(A_k,\beta)$. We
  consider the linear isomorphisms from
  Lemma~\ref{lemma:derivativeIso}, applied to $A_k$ (and \emph{not}
  $A$). Let $u\in \Z^k$ such that $A\cdot u = 0$ and $u_k=\ell\geq 0$. 
  We set $\psi_\ell = \del^{-\bar{u}}\psi$. Note that if $v\in
  \Z^k$ with $A\cdot v=0$ and $v_k=\ell$, then $A_k
  \cdot(\bar{u}-\bar{v})=0$. It follows that
  $\del^{\bar{u}-\bar{v}}\psi=\psi$, or equivalently
  $\del^{-\bar{u}}\psi=\del^{-\bar{v}}\psi$. In other words,
  $\psi_\ell$ does not depend on the choice of $u$. Note that
  $\psi_0=\psi$. 

  Given $\ell > 0$, if there is no $u \in \Z^k$ with $A\cdot u = 0$
  and $u_k=\ell$, we set $\psi_\ell = 0$. By construction,
  $\psi_\ell \in \Sol_x(A_k,\beta-A_k\cdot
  \bar{u})=\Sol_x(A_k,\beta-\ell \alpha_k)$.
  We define
  \[
    F = F(y_1,\dots,y_k) = \sum_{\ell=0}^\infty
    \frac{y_k^\ell}{\ell!}\,\psi_\ell(y_1, \dots, y_{k-1}).
  \]
  We claim that $F$ is a formal solution of $H_A(\beta)$. By
  construction, each summand of $F$ is annihilated by $\E-\beta$, so
  consider $v\in \Z^k$ and the operator $\del^{v_+}-\del^{v_-}$. If
  $v_k=0$, then this operator belongs to $I_{A_k}$ and, therefore,
  annihilates every summand of $F$. 
  
  Assume now that $v_k \neq 0$, say $v_k>0$.
  Let $u \in \Z^k$ such that $A\cdot u = 0$ and $u_k=\ell\geq
  v_k$. Then $(u-v)_k = \ell-v_k$ and, hence,
\[
  \psi_{\ell-v_k} = \del^{-\overline{(u-v)}}\psi =
  \del^{-\bar{u}+\bar{v}}\psi = \del^{-\bar{v}_-} \del^{\bar{v}_+}
    \del^{-\bar{u}}\psi = \del^{-\bar{v}_-} \del^{\bar{v}_+}\psi_\ell.
\]
  It follows that $\del^{\bar{v}_+}  \psi_\ell = \del^{\bar{v}_-}
  \psi_{\ell-v_k}$. Now,
\[
    \del^{v_+} F   = \sum_{\ell\geq v_k}
                    \frac{y_k^{\ell-v_k}}{(\ell-v_k)!}
                    \del^{\bar{v}_+}\psi_\ell 
                  = \sum_{\ell\geq v_k}
                   \frac{y_k^{\ell-v_k}}{(\ell-v_k)!}
                  \del^{\bar{v}_-}\psi_{\ell-v_k} 
                   = \del^{\bar{v}_-}
                    \sum_{\ell\geq v_k}
                   \frac{y_k^{\ell-v_k}}{(\ell-v_k)!}
                   \psi_{\ell-v_k}  
                  = \del^{v_-} F.
\]
  We have verified that $F$ is a formal solution of $H_A(\beta)$. By
  construction, $F(x,0) = \psi(x)$,
  so that, in particular,  $F$ is defined at $(x,0)$. 

  If  $\psi$ is given as a Nilsson series converging in a neighborhood
  of $x \in \C^A$, then $F$ is by construction a formal Nilsson
  series. Since $\cD/H_A(\beta)$ is a regular holonomic $\cD$-module, and
  $F(x,0)$ is defined, we may assume that $F$ converges in an open
  neighborhood of $(x,0)$ (after perturbing $x$ if necessary).
  Thus, if we apply this construction to the elements a basis of
  $\Sol_x(A_k,\beta)$, then we obtain a set of solutions of
  $H_A(\beta)$ with a common domain of convergence in a neighborhood
  of $(x,0)$, whose images under $y_k \mapsto 0$ are the elements of
  the basis of $\Sol_x(A_k,\beta)$. It follows that the linear
  transformation
  $\Psi: \Sol_{(x,0)}(A,\beta) \to \Sol_x(A_k,\beta)$ is
  surjective. Since both vector spaces involved have dimension
  $\vol(A)$, we conclude that $\Psi$ is an isomorphism.
\end{proof}

\begin{corollary}
\label{cor:MainMonodromyInductionStep}
Let $\alpha_k \in A$ be lattice redundant, and assume that there is
a face $\Gamma \in \cF_{\intt}$, containing $\alpha_k$,
which has a relative interior point $\alpha_1$ distinct form $\alpha_k$. 
Let $\beta$ be nonresonant. If $x$ is a generic nonsingular point of $H_{A_k}(\beta)$, so that $(x,0)$ is a
nonsingular point of $H_A(\beta)$, then
\[
\Mon_{(x,0)}(A, \beta) =  \Mon_x(A_k, \beta).
\]
\end{corollary}

\begin{proof}
Let $r = \vol(A)$ denote the lattice volume of $A$. Since $\beta$ is nonresonant,
the solution space $\Sol_{(x,0)}(A, \beta)$ has dimension $r$.
By Theorem~\ref{thm:SolutionSpaceIsomorphism}, we can choose
a basis $F_1, \dots, F_r$ of $\Sol_{(x,0)}(A, \beta)$ such 
that $F_1(\bar y,0), \dots, F_r(\bar y,0)$ is a basis for the solution
space of $\Sol_x(A_k, \beta)$. By Proposition~\ref{prop:InductionStep}, we find a set 
$G \subset \pi_1\big(\C^{A_k} \setminus V_{A_k}, x\big)$ which generates
both $\pi_1\big(\C^{A_k} \setminus V_{A_k}, x\big)$ and $\pi_1\big(\C^A \setminus V_A, (x,0)\big)$.
It follows that both $\Mon_{(x,0)}(A, \beta)$ and $\Mon_x(A, \beta)$ are generated
by monodromy matrices obtain by analytic continuation of the functions $F_1, \dots, F_r$
along the paths $\gamma \in G$. Hence, the two monodromy groups have the same set
of generators.
\end{proof}

\begin{theorem}
\label{thm:MainMonodromy2}
If $\beta$ is a nonresonant parameter, then 
\[
\Mon_{x}(A, \beta) \simeq \Mon_{(x,0)}(A^s, \beta).
\]
\end{theorem}

\begin{proof}
The proof is a simple induction using Corollary~\ref{cor:MainMonodromyInductionStep}.
\end{proof}

\begin{proof}[Proof of Theorem~\ref{thm:MainMonodromy}]
If $A$ has a relative interior point, then $A^p = A^s$.
Hence, the statement follows from Theorem~\ref{thm:MainMonodromy2}.
\end{proof}

\section{Monomial Curves}
\label{sec:Lines}

Let us exemplify our main theorems by considering a 
one-dimensional collection
\[
A = \left[\begin{array}{ccccc}  1 & 1 & \dots & 1 & 1\\
0 & \alpha_1 & \dots & \alpha_m  & \delta
\end{array}\right],
\]
where $0 < \alpha_1 < \dots < \alpha_m < \delta$. The toric variety associated to such a one-dimensional collection is called a \emph{monomial curve}. The integer $\delta$ is the \emph{toric degree}
of $A$. There is no loss of generality in assuming that $\gcd(\alpha_1, \dots, \alpha_m, \delta) = 1$,
so that $\Z_A = \Z^2$. In this case,
the partial face saturation and the face saturation of $A$ both coincide with 
the saturation $N\cap \Z_A$.
The space $(\C^*)^A$ is identified with the space of all univariate polynomials
\[
f(z) = y_0 + y_1 z^{\alpha_1} + \dots  + y_m z^{\alpha_m} + y_{m+1} z^\delta.
\]
The reduced principal $A$-determinant is 
$\widehat E_A(y) = y_0 y_{m+1} D_A(y)$,
where $D_A(y)$ is the $A$-discriminant. In particular, there is a map
\[
V\colon \C^A \setminus E_A \rightarrow C_\delta(\C^*),
\]
where $C_\delta(\C^*)$ denotes the configuration space
of $\delta$ distinct points in $\C^*$. Taking fundamental groups,
we obtain the \emph{braid map}
\[
\br\colon \pi_1\big(\C^A \setminus E_A, x\big) \rightarrow \CB_\delta,
\]
where $\CB_\delta$ denotes the cyclic braid group
on $\delta$ strands.

\begin{corollary}
The braid map $\br$ is surjective.
\end{corollary}

\begin{proof}
If $A$ is saturated, then there is a map $C_\delta(\C^*) \rightarrow \C^A \,\setminus E_A$
which takes a set of $\delta$ points to the monic polynomial of degree $\delta$ vanishing on said set.
This fact, and Theorem~\ref{thm:MainFundamentalGroup}, imply that the
braid map is surjective also in the general case.
\end{proof}

The braid map is not injective.  This is primarily an effect of working with the
principal $A$-determinant rather then the $A$-discriminant. For example,
the cycle 
\[
t \mapsto e^{2 \pi i t} y
\]
is nontrivial (and acts diagonally, yet nontrivially, in monodromy) in $\C^A \setminus E_A$.
We conclude that there is a surjective map
\[
\bar \br\colon \pi_1\big(\C^A \setminus E_A, x\big) \rightarrow \Z \times \CB_\delta.
\]
There are simple arguments, using for example the order map of the coamoeba \cite{FJ15}, 
which imply that the extended braid map $\bar\br$ is injective whenever $A$ consists of three 
elements that generate $\Z$. It follows that the braid map is injective whenever there exists an $i \in \{1, \dots, m\}$ such that $\gcd(a_i, \delta) = 1$. While we expect that the extended  braid map is aways injective,
no such simple argument is available in the general case.

\smallskip
In terms of computing the monodromy group of the $A$-hypergeometric system,
it is beneficial to choose $A$ with as few points as possible.
The monodromy group $\Mon_x(A, \beta)$ only depends on
$\beta$ and the toric degree $\delta$. Indeed, the collection $A$ has the same saturation as
the triple
\[
A' = \left[\begin{array}{ccc}  1 & 1 & 1\\
0 & 1 & \delta
\end{array}\right],
\]
which is of codimension one and, hence, admits a Mellin--Barnes basis.
Using the torus action, we find that $\C^{A'}\setminus E_{A'}$ is a trivial bundle,
with fibers $(\C^*)^2$, and base $\C \setminus D_{A'}$, where 
$D_{A'}$ is the dehomogenized discriminant defined as the intersection
of $D_{A'}$ with the linear subspace $x_1 = x_{2} = 1$.
We get a long exact sequence
\begin{center}
\begin{tikzcd}
\dots \arrow[r]  
& \Z^2\arrow[r] 
& \pi_1\big(\C^{A'}\,\setminus E_{A'}, x\big) \arrow[r] 
& \pi_1\big(\C \,\setminus D_{A'}, x\big) \arrow[r]
& 0.
\end{tikzcd}
\end{center}
The dehomogenized discriminant $D_{A'}$ is a rational zero-dimensional variety and, hence,
 consists of a single point \cite{Kap91}.
It follows that $\pi_1(\C^{A'} \setminus E_{A'}, x)$ has three generators, in this case. 
The corresponding monodromy matrices can be computed using Beukers' method \cite{Beu16}.
For example, if $\delta = 3$, then the monodromy group is generated by
\[
\left[\begin{array}{ccc}
e^{2\pi i \beta_1} & 0 & 0 \\
0 & e^{2\pi i \beta_1} & 0\\
0 & 0 & e^{2\pi i \beta_1}
\end{array}\right],
\quad
\left[\begin{array}{ccc}
0 & 1 & 0 \\
0 & 0 & e^{2\pi i (\beta_2-\beta_1)}\\
e^{2\pi i \beta_1} & 0 & 0
\end{array}\right],
\quad \text{ and } \quad
\left[\begin{array}{ccc}
-1 & 1 & 1 \\
0 & e^{2\pi i \beta_1} & 0\\
0 & 0 & e^{2\pi i \beta_1}
\end{array}\right].
\]

\bibliographystyle{amsplain}

\end{document}